\documentclass{article}
\usepackage[utf8]{inputenc}
\usepackage{commands}
\usepackage[margin=1in]{geometry}
\usepackage[title]{appendix}

\title{Erd\H os--Littlewood--Offord problem with arbitrary probabilities}
\author{Mihir Singhal\thanks{Department of Mathematics, Massachusetts Institute of Technology, Cambridge, MA 02139. Email: \href{mailto:mihirs@mit.edu}{\nolinkurl{mihirs@mit.edu}}.}} 
\date{}

\begin{document}

\maketitle

\begin{abstract}
The classical {\elo} problem concerns the random variable $X = a_1 \xi_1 + \dots + a_n \xi_n$, where $a_i \in \R \setminus \{0\}$ are fixed and $\xi_i \sim \Ber(1/2)$ are independent. The {\elo} theorem states that the maximum possible concentration probability $\max_{x \in \R} \Pr(X = x)$ is $\binom{n}{\lfloor n/2\rfloor} / 2^n$, achieved when the $a_i$ are all 1. As proposed by Fox, Kwan, and Sauermann, we investigate the general case where $\xi_i \sim \Ber(p)$ instead. Using purely combinatorial techniques, we show that the exact maximum concentration probability is achieved when $a_i \in \{-1, 1\}$ for each $i$. Then, using Fourier-analytic techniques, we investigate the optimal ratio of 1s to $-1$s. Surprisingly, we find that in some cases, the numbers of 1s and $-1$s can be far from equal.
\end{abstract}

\section{Introduction}

The Erd\H os--Littlewood--Offord theorem is a classical result in combinatorics and probability theory. It concerns anticoncentration of random variables of the form
\begin{equation*}
X = \sum_{i=1}^n a_i \xi_i,
\end{equation*}
where $a_1, \dots, a_n$ are nonzero reals and the $\xi_i$ are independent instances of the Bernoulli random variable $\Ber(1/2)$, which takes the values $0$ and $1$ each with probability $1/2$. 
The {\elo} theorem \cite{elo} asserts that the \textit{concentration probability} of $X$, defined as $\max_{x \in \R} \Pr(X = x)$, is at most $\binom{n}{\lfloor n/2 \rfloor} / 2^n = O(1/\sqrt n)$. This is tight since equality is achieved when the $a_i$ are $\pm 1$. 

It is natural to look at the behavior of $X$ when the $\xi_i$ follow a general Bernoulli distribution $\Ber(p)$ instead of $\Ber(1/2)$, meaning that they take the values $0$ and $1$ with probabilities $1-p$ and $p$, respectively. 
In \cite{cacia}, Fox, Kwan, and Sauermann asked what bounds can be given on the concentration probability as a function of $n$ and $p$, and showed some related asymptotic bounds. The specific question we are concerned with in this paper is the following.

\begin{question} \label{ques} \normalfont
Let $n$ be any positive integer and let $p$ be any probability. Define the random variable
\[X = \sum_{i=1}^n a_i \xi_i,\]
where $a_1, \dots, a_n$ are nonzero reals and $\xi_1, \dots, \xi_n$ are independent instances of $\Ber(p)$. What is the maximum possible value of the concentration probability $\max_{x \in \R} \Pr(X = x)$, as a function of $n$ and $p$?
\end{question}

If all the $a_i$ are positive, then one can imitate Erd\H os's original proof of the {\elo} theorem to show that the maximum is achieved when $a_1 = \dots = a_n$. In the classical case where $p = 1/2$, it makes no difference to assume all the $a_i$ are positive, since negating any $a_i$ only shifts $X$ by a constant (in distribution). However, one cannot make such an assumption in general, making the general case much harder.

In this paper, we show the following result, which provides an exact bound for all $n$.

\begin{thm} \label{mainthm}
Fix a positive integer $n$ and real $p$ between 0 and 1. Let $a_1, \dots, a_n$ be nonzero reals, and let $\xi_1, \dots, \xi_n$ be independent instances of $\Ber(p)$. Consider the random variable
\[X = \sum_{i=1}^n a_i \xi_i.\]
Then, the concentration probability $\max_{x \in \R} \Pr(X = x)$ is maximized when all the $a_i$ are $\pm 1$. In other words,
\begin{equation} \label{binbd}
    \Pr(X = x) \le \max_{0 \le \ell \le n} \max_{y \in \Z} \Pr(\Bin(\ell, p) - \Bin(n - \ell, p) = y).
\end{equation}
Here $\ell$ corresponds to the number of $a_i$ that equal 1.
\end{thm}

Perhaps the most standard tool for studying sums of independent random variables is Fourier analysis. We remark that Fourier-analytic techniques do provide a partial answer to the question. For even $n$, Fourier analysis (as in \cite[Proposition 6.1]{cacia})
yields the result of Theorem \ref{mainthm}, with the optimum occurring when $\ell = \frac n2$. It fails to give a bound that works for all $n$, however -- for odd $n$, the best it can show is that the concentration probability is at most $1 + o(1)$ times the probabilty achieved when $(n+1)/2$ of the $a_i$ are 1 and the rest are $-1$. 

Like Erd\H os's original proof of the {\elo} theorem \cite{elo}, our proof uses only the trivial fact that if $T \subsetneq S$ are sets of positive numbers, then the sum of the elements of $T$ is less than the sum of the elements of $S$.
In our proof, we split the $a_i$ into positive and negative parts. We then consider vectors $\alpha, \beta$ which are two-argument functions, where $\alpha(k, r)$ is the probability that a random $k$-subset of the positive $a_i$ sums to $r$, and $\beta$ is analogously defined for negative $a_i$. We observe that we can write the point probability $\Pr(X = x)$ as a bilinear function of $\alpha, \beta$. 
We then show that $\alpha, \beta$ can always be expressed as a convex combination of ``pure" objects (which are of a very specific structure). Since $\Pr(X = x)$ is bilinear in $\alpha, \beta$, we then upper bound it by its value when $\alpha, \beta$ are pure. This leaves us with a simple expression that we are able to bound explicitly.

Theorem \ref{mainthm} mostly resolves Question \ref{ques}. However, determining $\ell^*$, the number of $a_i$ that are 1 in the maximal case, is surprisingly nontrivial. One might expect that, $\ell = \lceil n/2 \rceil$ maximizes the expression (\ref{binbd}) in Theorem \ref{mainthm}. Indeed, there is some precedent for this kind of situation: in \cite{modq}, Vaughan and Wooley studied a mod $q$ analog to this problem, finding that in their setting, the worst case occurs when $\lfloor n/2 \rfloor$ of the coefficients are $-1$ and $\lceil n/2 \rceil$ are $1$.

However, surprisingly, the optimal $\ell$ in Theorem \ref{mainthm} may be far from $n/2$. We discuss this problem further in Section \ref{maxl}. We use the result of Theorem \ref{mainthm} (that all the $a_i$ are $\pm 1$) in order to greatly simplify certain Fourier-analytic expressions for the probability, allowing us to obtain fairly tight bounds on $\Pr(X = x)$. We then obtain asymptotic answers for $\ell^*$ as $n$ grows large, for fixed $p$. These results are tabulated later in Figure \ref{fig:lstarresults}, in Section \ref{concl}.

\subsection*{Notation}
We use standard asymptotic notation throughout, including $O, \Omega, o, \omega$. All asymptotic notation is to be taken as $n \rightarrow \infty$. For example, $f = o(g)$ if $\lim_{n \rightarrow \infty} f/g = 0$, uniformly in all parameters other than $n$. Some expressions may also be negative; when we say $f = O(g)$ or $f = o(g)$, we mean that $|f| = O(g)$ or $|f| = o(g)$, respectively. On the other hand, when we say $f = \Omega(g)$ or $f = \omega(g)$ we require that $f$ be positive for sufficiently large $n$.

We use $\Ber(p)$ and $\Bin(n, p)$ to denote the Bernoulli and binomial distributions, respectively. As in the statement of Theorem \ref{mainthm}, we occasionally abuse notation by using these to denote random variables with the same distribution; different instances occurring in a single expression are intended to be independent.

We also frequently write sums over parameters that can take infinitely or uncountably many values. In such cases, the summand will only be nonzero for finitely many values of the parameter, and the sum will be understood to mean the sum of only those nonzero values.

\section{Proof of Theorem \ref{mainthm}} \label{mainpf}

Without loss of generality, suppose $a_1, \dots, a_\ell$ are positive, and $a_{\ell+1}, \dots, a_n$ are negative. Define the multisets $A = \{a_1, \dots, a_\ell\}, B = \{-a_{\ell+1}, \dots, -a_n\}$; then $|A| = \ell, |B| = m$, where $m = n - \ell$. We will henceforth treat $\ell, m$ as fixed. 
Define the (nonnegative) random variables $Y = \sum_{i = 1}^\ell a_i\xi_i$ and $Z = -\sum_{i = \ell+1}^n a_i\xi_i$. Note that $X = Y - Z$. 

Now, we define the function $\alpha: \{0, \dots, \ell\} \times \R \rightarrow \R$ so that $\alpha(k, r)$ is the probability that a uniformly random $k$-subset of $A$ sums to $r$. Similarly define $\beta: \{0, \dots, m\} \times \R \rightarrow \R$ so that $\beta(k, r)$ is the probability that a uniformly random $k$-subset of $B$ sums to $r$. Note that $\alpha, \beta$ have finite support (where the support is defined as the set of $(k, r)$ at which $\alpha$ or $\beta$ is nonzero). 

By conditioning on the number of $\xi_1, \dots, \xi_\ell$ that equal $1$, we have that
\[\Pr(Y = r) = \sum_{k=0}^\ell \binom{\ell}{k} p^k (1-p)^{\ell-k} \alpha(k, r),\]
and similarly, 
\[\Pr(Z = r) = \sum_{k=0}^m \binom{m}{k} p^k (1-p)^{m-k} \beta(k, r).\] Note that $\Pr(Y = r)$ depends only on $\alpha$ and $r$, and is a linear function of $\alpha$ (where the space of functions from $\{0, \dots, \ell\} \times \R$ to $\R$ with finite support is treated as a vector space in the obvious way). Similarly, $\Pr(Z = r)$ depends only on $\beta$ and $r$, and is a linear function of $\beta$. Then, 
\begin{align}
    \Pr(X = x) &= \sum_r \Pr(Y = r)\Pr(Z = r - x) \nonumber \\
    &= \sum_r \paren{\sum_{k=0}^\ell \binom{\ell}{k} p^k (1-p)^{\ell-k} \alpha(k, r)}\paren{\sum_{k=0}^m \binom{m}{k} p^k (1-p)^{m-k} \beta(k, r - x)} \label{ptprob}
\end{align}
is a bilinear function of $\alpha, \beta$. Define $B(\alpha, \beta)$ to be the bilinear function in (\ref{ptprob}).

We now note some properties of $\alpha, \beta$.

\begin{fact} \label{props}
The following are true when $\gamma$ is $\alpha$ or $\beta$:
\begin{enumerate}
    \item $\gamma(k, r) \ge 0$ for all $k, r$. 
    \item $\sum_r \gamma(k, r) = 1$ for all $k$.
    \item For all $k, r$,
\begin{equation*}
\sum_{r' \le r} \gamma(k+1, r') + \sum_{r' \ge r} \gamma(k, r') \le 1.
\end{equation*}
\end{enumerate}
\end{fact}

\begin{proof}
Properties 1 and 2 follow immediately from the definition of $\alpha, \beta$. We will prove property 3 for $\gamma = \alpha$; the $\gamma = \beta$ case is identical.

Let $\ms{S}$ be the collection of $k$-subsets (subsets of size $k$) of $A$ whose sum is at least $r$ and let $\ms{T}$ be the collection of $(k+1)$-subsets of $A$ whose sum is at most $r$. We wish to show that $|\ms{T}|/\binom{\ell}{k+1} + |\ms{S}|/\binom{\ell}{k} \le 1$. For a contradiction, suppose the opposite. Then, multiplying through, we must have
\begin{equation}\label{STsizes}
|\ms T| + \frac{\ell-k}{k+1}|\ms S| > \binom{\ell}{k+1}.
\end{equation}
Now, consider the bipartite graph which connects subsets of $A$ of size $k$ (the left side) and $k+1$ (the right side) by inclusion. This is a biregular graph, with degree $\ell - k$ on the left and degree $k + 1$ on the right. Thus, the number of edges touching $\ms S$ is $(\ell - k)|\ms S|$, and so the size of $N(\ms S)$, the set of vertices that neighbor an element of $\ms S$, is at least $\frac{\ell-k}{k+1}|\ms S|$. Thus, by (\ref{STsizes}), $N(\ms S)$ and $\ms T$ have an element in common, which means that there exist $S \in \ms S$ and $T \in \ms T$ such that $S \subset T$. But then the sum of the elements of $T$ would be greater than that of $S$, contradicting the fact that the sum of the elements of $T$ is at most $r$ and the sum of the elements of $S$ is at least $r$.
\end{proof}

Let $\mcF_i$ be the set of all functions from $\{0, \dots, i\} \times \R$ to $\R$ with finite support satisfying the three properties of Fact \ref{props}. (We henceforth refer to the parts of Fact \ref{props} as Properties 1, 2, 3, respectively.) Then $\alpha \in \mcF_\ell, \beta \in \mcF_m$.

Define $\gamma \in \mcF_i$ to be \textit{pure} if there exist $r_0 < r_1 < \dots < r_i$ so that $\gamma(k, r) = 1$ if $r = r_k$, and otherwise $\gamma(k, r) = 0$.

\begin{lem} \label{convex}
All $\gamma \in \mcF_i$ are a convex combination of pure functions.
\end{lem}
\begin{proof}
We induct on the size of the support of $\gamma$. By Property 2 (of Fact \ref{props}), $\gamma$ must have support with size at least $i + 1$. If the support has size $i + 1$, then Properties 2 and 3 imply that $\gamma$ itself must be pure.

Now, suppose the hypothesis is true for functions with support whose size is less than that of $\gamma$ (and suppose $\gamma$ has support with size strictly larger than $i + 1$). For each $k$, let $r_k$ be the minimal $r$ so that $\gamma(k, r) > 0$. Furthermore, let $\lambda = \min_k \gamma(k, r_k)$. We have that $r_k < r_{k+1}$ for all $k$, since otherwise, we would have \[\sum_{r' \le r_k} \gamma(k+1, r') + \sum_{r' \ge r_k} \gamma(k, r') \ge \gamma(k+1, r_{k+1}) + 1 > 1,\] contradicting Property 3. Therefore, we can construct the pure function $\zeta$ such that $\zeta(k, r_k) = 1$ for all $k$.

Let $\gamma' = \frac{1}{1-\lambda} (\gamma - \lambda \zeta)$. ($\lambda < 1$ by Property 2 since $\gamma$ has support with size greater than $i + 1$.) Since $\gamma$ is a convex combination of $\gamma'$ and $\zeta$, and $\gamma'$ has support strictly smaller than that of $\gamma$, by the inductive hypothesis it suffices to check that $\gamma' \in \mcF_i$, or in other words, that $\gamma'$ satisfies Fact \ref{props}. By definition of $\lambda$, $\gamma'$ is nonnegative everywhere, satisfying Property 1. Property 2 follows from the fact that Property 2 is satisfied by $\gamma$ and $\zeta$. It remains to check Property 3. If $r < r_{k+1}$, then

\begin{align*}
    \sum_{r' \le r} \gamma'(k+1, r') + \sum_{r' \ge r} \gamma'(k, r') \le 0 + 1 \le 1,
\end{align*}

due to the fact that $\gamma'(k+1, r')$ is zero when $r' < r_{k+1}$ and that $\gamma'$ satisfies Property 2. On the other hand, if $r \ge r_{k+1}$, then

\begin{align*}
\sum_{r' \le r} \gamma'(k+1, r') + \sum_{r' \ge r} \gamma'(k, r') 
&= \frac{1}{1-\lambda} \paren{\sum_{r' \le r} \gamma(k+1, r') - \lambda} + \frac{1}{1-\lambda} \paren{\sum_{r' \ge r} \gamma(k, r')} \\
&= \frac{1}{1-\lambda} \paren{\sum_{r' \le r} \gamma(k+1, r') + \sum_{r' \ge r} \gamma(k, r') - \lambda} \\
&\le \frac{1}{1-\lambda} \paren{1 - \lambda} \\
&= 1,
\end{align*}

where the inequality is because $\gamma$ satisfies Property 3. Thus $\gamma' \in \mcF_i$, and we are done.

\end{proof}

By Lemma \ref{convex}, since $B(\alpha, \beta)$ (from (\ref{ptprob})) is bilinear in $\alpha, \beta$, and $\alpha, \beta$ are convex combinations of pure functions, $B(\alpha, \beta)$ is bounded above by its value at some pure functions $\alpha' \in \mcF_\ell, \beta' \in \mcF_m$. In other words, we have
\begin{align}
\Pr(X = x) &= B(\alpha, \beta) \nonumber \\
&\le B(\alpha', \beta') \nonumber \\
&=
\sum_r \paren{\sum_{k=0}^\ell \binom{\ell}{k} p^k (1-p)^{\ell-k} \alpha'(k, r)}\paren{\sum_{k=0}^m \binom{m}{k} p^k (1-p)^{m-k} \beta'(k, r - x)}. \label{ptprobprime}
\end{align}

Since $\alpha', \beta'$ are pure, there must be $r_0 < \dots < r_\ell$ and $s_0 < \dots < s_m$ so that $\alpha'(k, r_k) = 1$ and $\beta'(j, s_j) = 1$ for all $k, j$ (and $\alpha', \beta'$ are 0 everywhere else). Then, we may evaluate (\ref{ptprobprime}) to get

\begin{align*} 
\Pr(X = x) &\le \sum_{\substack{k, j \\ r_k = s_j + x}} \binom{\ell}{k} p^k (1-p)^{\ell-k} \binom{m}{j} p^j (1-p)^{m-j}.
\end{align*}

Defining $f(k) = \Pr(\Bin(\ell, p) = k) = \binom{\ell}{k} p^k (1-p)^{\ell-k}$ and $g(j) = \Pr(\Bin(m, p) = j) = \binom{m}{j} p^j (1-p)^{m-j}$, this becomes
\begin{align} \label{ptprob2}
\Pr(X = x) &\le \sum_{\substack{k, j \\ r_k = s_j + x}} f(k) g(j).
\end{align}

Since this is an inequality, we may further assume that $r_0 < \dots < r_\ell$ and $s_0 < \dots < s_m$ are such that the right hand side of (\ref{ptprob2}) is maximized. (The maximum exists because the right hand side of (\ref{ptprob2}) can only take on finitely many values).

Now, our goal is to show that for these maximizing $r_0 < \dots < r_\ell$ and $s_0 < \dots < s_m$, the right hand side of (\ref{ptprob2}) is actually equal to $\Pr(\Bin(\ell, p) - \Bin(m, p) = d)$. To this end, in the following claims we prove some simple facts about the $r_i$ and $s_j$.

\begin{claim} \label{interleave}
It is never the case that $r_k < s_j + x < r_{k+1}$ or $s_j + x < r_k < s_{j+1} + x$ for any $k, j$.
\end{claim}
\begin{proof}
Suppose for a contradiction that $r_{k_0} < s_{j_0} + x < r_{k_0+1}$ for some $k_0, j_0$. (We only prove the first part; the other part is identical.) Note that $g$ is unimodal, so it must be the case that either $g(j_0) > g(j_0-1) > \dots > g(0)$ or $g(j_0) > g(j_0+1) > \dots > g(m)$. Suppose that $g(j_0) > g(j_0-1) > \dots > g(0)$; the other case is similar (the situation is symmetric with respect to flipping the order of all indices). We now split into cases based on whether there exists $j$ such that $s_j + x = r_{k_0}$.

\textbf{Case 1:} There is no $j$ such that $s_j + x = r_{k_0}$. Let $j'$ be minimal such that $s_{j'} + x > r_{k_0}$. There is no $k$ such that $s_{j'} + x  = r_k$, since $s_{j'} + x \le s_{j_0}+x < r_{k_0+1}$. Then let $s'_{j'} = r_{k_0}-x$, and $s'_i = s_i$ for $i \neq j$. By minimality of $j'$, this preserves the condition $s'_0 < \dots < s'_m$. Also, replacing $s_i$ with $s'_i$ increases (\ref{ptprob2}) by $f(k_0)g(j') > 0$, contradicting maximality of (\ref{ptprob2}).

\textbf{Case 2:} There exists $j$ such that $s_j + x = r_{k_0}$. Note that $j < j_0$, so $s_{j+1} + x \le s_{j_0} + x < r_{k_0+1}$, so there is no $k$ such that $s_{j+1} + x = r_k$. Then, define $s'_{j+1} = s_j = r_{k_0} - x$, and let $s'_j$ be an arbitrary value between $s_{j-1}$ and $s_j$. Also let $s'_i = s_i$ for all $i \notin \{j, j+1\}$. This clearly preserves the condition $s'_0 < \dots < s'_m$, and the the value of (\ref{ptprob2}) evaluated at $s'_j$ instead of $s_j$ exceeds its original value by at least $f(k_0)(g(j+1) - g(j)) > 0$. This again contradicts maximality of (\ref{ptprob2}).
\end{proof}

\begin{claim} \label{matching}
Suppose that $r_k = s_j + x$ for some $k$. Then, for all $0 \le k' \le \ell$ and $0 \le j' \le m$ such that $k' - j' = k - j$, we also have $r_{k'} = s_{j'} + x$. 
\end{claim}
\begin{proof}
First we show the statement for $k' \ge k, j' \ge j$ by induction. The base case $k'=k$ and $j'=j$ is given. Assume that $0 \le k' \le \ell$ and $0 \le j' \le m$ such that $k' - j' = k - j$, and suppose the hypothesis holds for $k' - 1$ and $j' - 1$. Since $r_i, s_j$ are increasing, we have that $r_{k'}$ and $s_{j'} + x$ are greater than $r_{k'-1} = s_{j'-1}+x$. By Claim \ref{interleave}, we cannot have $s_{j'-1}+x < r_{k'} <  s_{j'} + x$ or $r_{k'-1} < s_{j'}+x < r_{k'}$, so it must be the case that $r_{k'} = s_{j'}+x$, so the induction is complete.

Similarly, for $k' \le k$ and $j' \le j$, we induct downward. Again suppose that $0 \le k' \le \ell$ and $0 \le j' \le m$ such that $k' - j' = k - j$, and suppose the hypothesis holds for $k' + 1$ and $j' + 1$. Then $r_{k'}, s_{j'} + x$ are less than $r_{k'+1} = s_{j'+1} + x$. By Claim \ref{interleave}, we cannot have $s_{j'}+x < r_{k'} < s_{j'+1} + x$ or $r_{k'} < s_{j'}+x < r_{k'+1}$. Thus we must again have $r_{k'} = s_{j'}+x$, completing the induction. 
\end{proof}

From Claim \ref{matching}, it follows that the right hand side of (\ref{ptprob2}) takes the form
\begin{align*} 
\sum_{\substack{0 \le k \le \ell \\ 0 \le j \le m \\ k = j + d}} f(k) g(j).
\end{align*}
for some (not necessarily positive) integer $d$. But recalling the definitions of $f,g$, this is just equal to $\Pr(\Bin(\ell, p) - \Bin(m, p) = d)$, giving the desired upper bound, thus completing the proof of Theorem \ref{mainthm}.

\section{Determining the maximal \texorpdfstring{$\ell$}{l}} \label{maxl}
In this section we discuss the problem of determining the value of $\ell$ which maximizes the right hand side of (\ref{binbd}). By Theorem \ref{mainthm}, we assume throughout this section that $\ell$ of the $a_i$ are 1 and the other $m$ are $-1$ (where $\ell + m = n$). 

Let $\ell^*$ be a value of $\ell$ that maximizes the concentration probability $\max_{x \in \R} \Pr(X = x)$, and let $x^*$ be the optimal $x$. These may take multiple possible values; we will abuse notation by saying that $\ell^* = \ell, x^* = x$ to mean that $\ell, x$ are one choice that maximizes $\Pr(X = x)$.

\subsection{\texorpdfstring{$n$}{n} even}

When $n$ is even, one can use Fourier analysis to show that the concentration probability is maximized when half the $a_i$ are $1$ and the other half are $-1$. We use Theorem \ref{mainthm} to simplify the analysis, but this result can also be shown with only Fourier analysis similarly to \cite[Proposition 6.1]{cacia}.

\begin{thm} \label{evencase}
If $n$ is even, then $\ell^* = n/2, x^* = 0$. Equivalently,
\begin{align*}
\Pr(X = x) &\le \Pr(\Bin(n/2, p) - \Bin(n/2, p) = 0) \\
&= \sum_{k=0}^{n/2} \binom{n/2}{k} p^{2k}(1-p)^{n-2k}.
\end{align*}
\end{thm}
\begin{proof}
By Theorem \ref{mainthm}, it suffices to consider the case where $\ell$ of the $a_i$ are 1 and $m = n - \ell$ are $-1$. Let $N$ be a prime greater than $2n$; we will do Fourier analysis over $\Z/n\Z$. Since $-n \le X \le n$, the event $X = x$ is equivalent to the event $X \equiv x \pmod N$, for $-n \le x \le n$.

Now let $f = (1-p)\delta_0 + p\delta_1, g = (1-p)\delta_0 + p \delta_{-1}$ be the probability mass functions of $\xi_i, -\xi_i$, respectively. They have corresponding Fourier transforms
\begin{align*}
\hat f(k) &= (1-p) + pe^{-2\pi ik/N}, \\
\hat g(k) &= (1-p) + pe^{2\pi ik/N}.
\end{align*}
Note that $\hat f, \hat g$ are complex conjugates of each other. We then have, by Fourier inversion,
\begin{align*}
\Pr(X = x) &= f^{*\ell} * g^{*m}(x) \\
&= \frac 1N \sum_{k=0}^{N-1} e^{2\pi ixk/N} \hat f(k)^\ell \hat g(k)^m \\
&\le \frac 1N \sum_{k=0}^{N-1} |\hat f(k)|^\ell |\hat g(k)|^m \\
&= \frac 1N \sum_{k=0}^{N-1} |\hat f(k)|^n,
\end{align*}
and the last expression here is a constant that does not depend on $\ell$. Note that when $\ell = m = n/2$ and $x = 0$, the inequality is actually an equality, since all the terms of the sum are positive and real due to $\hat f$ and $\hat g$ being complex conjugates. Thus, $\Pr(X = x)$ is maximized when $\ell = n/2$, as desired.
\end{proof}

\subsection{\texorpdfstring{$p$}{p} small}

The case where $p$ is very small relative to $n$ is also easy, since the optimal $x$ will necessarily be 0, which allows us to use a convexity argument to find $\ell^*$.

\begin{lem}
If $p \le 1 - (1/2)^{1/n} = O(1/n)$, then $\ell^* = \lceil n/2 \rceil$ and $x^* = 0$. In other words, $\Pr(X = x)$ is maximized when $a_1, \dots, a_{\lceil n/2 \rceil} = 1$ and the other coefficients are $-1$, and $x = 0$.
\end{lem}

\begin{proof}
Note that regardless of the choice of $\ell^*$, we have that $\Pr(X = 0) \ge \Pr(\xi_i = 0 \text{ for all } i) = \frac 12$. Thus, the highest point probability of $X$ must be at $0$, so $x^* = 0$. It remains to determine the value of $\ell$ that maximizes $\Pr(X = 0)$. 

Assume without loss of generality that $a_1, \dots, a_\ell = 1$ and the rest are $-1$. Then, we may compute $\Pr(X = 0)$ by casework on the number of positive and negative terms in the expression for $X$:

\begin{align}
\Pr(X = 0) &= \Pr(\# \{1 \le i \le \ell: \xi_i = 1\} = \# \{\ell < i \le n: \xi_i = 1\}) \nonumber \\
&= \sum_{k = 0}^n \Pr(\# \{1 \le i \le \ell: \xi_i = 1\} = k) \Pr(\# \{\ell < i \le n: \xi_i = 1\} = k) \nonumber \\
&= \sum_{k = 0}^n \binom \ell k p^k (1-p)^{\ell-k} + \binom{n-\ell} k p^k (1-p)^{n - \ell - k} \nonumber \\
&= \sum_{k = 0}^n p^{2k}(1-p)^{n-2k} \binom \ell k \binom{n-\ell} k \nonumber \\
&= \sum_{k = 0}^n p^{2k}(1-p)^{n-2k}(k!)^{-2} (\ell(n-\ell)) ((\ell-1)(n - \ell - 1)) \dots ((\ell - k) (n - \ell - k)). \label{pr0exp}
\end{align}
Now note that each term of the form $(\ell - i)(n - \ell - i)$ is maximized when $\ell = \lceil n/2 \rceil$. Thus, the entire expression (\ref{pr0exp}) is maximized when $\ell = \lceil n/2 \rceil$, as desired.
\end{proof}

\subsection{\texorpdfstring{$n$}{n} odd} \label{odd}

Finding $\ell^*$ turns out to be significantly harder when $n$ is odd and $p$ is not vanishingly small. We will show some partial results.

We consider only the case where $p$ is fixed and $n$ is large. In what follows, we will assume that $p$ is a fixed constant for the purposes of asymptotic notation. We also assume $p \neq \frac 12$, since in the $p = \frac 12$ case, any value of $\ell$ yields the same probability. Let $t = \ell - m = 2\ell - n$, and without loss of generality suppose $t > 0$, or equivalently $\ell > m$ (else flip the signs of all the $a_i$ and $x$). We will consider the maximization problem in terms of $t, x$ instead of $\ell, x$. Note that $t$ can take on any odd value between 0 and $n$ (inclusive). 

In this regime, it will again be helpful to use Fourier analysis. Again let $N$ be an arbitrary prime greater than $2n$; we use Fourier analysis over $\Z / N\Z$. Defining $f$ and $g$ as in the proof of Theorem \ref{evencase}, 
\begin{align*}
\Pr(X = x) &= f^{*\ell} * g^{*m}(x) \\
&= \frac 1N \sum_{k=-(N-1)/2}^{(N-1)/2} e^{2\pi ixk/N} \hat f(k)^\ell \hat g(k)^m \\
&= \frac 1N \sum_{k=-(N-1)/2}^{(N-1)/2} e^{2\pi ixk/N} (1-p+pe^{-2\pi i k/N})^\ell (1-p+pe^{2\pi i k/N})^m.
\end{align*}
Taking the limit as $N \rightarrow \infty$ (taking prime values only) and noting that $\Pr(X = x)$ is real, we thus have
\begin{align}
\Pr(X = x) &= \frac{1}{2\pi} \int_{-\pi}^{\pi} e^{ixy} (1-p+pe^{-iy})^\ell  (1-p+pe^{iy})^m \, dy \nonumber \\
&= \frac{1}{2\pi} \int_{-\pi}^{\pi} \Re\left(e^{ixy} (1-p+pe^{-iy})^\ell  (1-p+pe^{iy})^m\right) \, dy \nonumber \\
&= \frac{1}{\pi} \int_{0}^{\pi} \Re\left(e^{ixy} (1-p+pe^{-iy})^\ell  (1-p+pe^{iy})^m\right) \, dy \nonumber \\
&= \frac{1}{\pi} \int_{0}^{\pi} 
\left| 1-p+pe^{-iy} \right|^n
\cos \left( xy + t \arg(1-p+pe^{-iy}) \right)
\, dy. \nonumber \\
&= \frac{1}{\pi} \int_{0}^{\pi} 
\left| 1-p+pe^{-iy} \right|^n \, dy 
\nonumber \\& \qquad -
\frac{1}{\pi} \int_{0}^{\pi} 
\left| 1-p+pe^{-iy} \right|^n
\left(1-\cos \left( xy + t \arg(1-p+pe^{-iy}) \right)\right)
\, dy. \label{qderiv}
\end{align}
The first integral on the right hand side depends only on $p$ and $n$ (and not $x$ or $t$), so maximizing $\Pr(X = x)$ is equivalent to minimizing the following function of $t, x$.
\begin{equation}
q(t, x) = 
\int_{0}^{\pi} 
\left| 1-p+pe^{-iy} \right|^n
\left(1-\cos \left( xy + t \arg(1-p+pe^{-iy}) \right)\right) 
\, dy. \label{qexp}
\end{equation}

Approximating $q(t, x)$ involves technical estimation; we defer it to Appendix \ref{app} and restate the results here. We henceforth assume that $p$ is not any of $0, \frac 12, 1$, since these cases are trivial (in all three cases, all choices for $\ell$ give the same concentration probability). From (\ref{qestapp}) and Lemma \ref{closetomean}, we have that if $x$ maximizes $\Pr(X = x)$ (or if $|x - tp| < n^{0.01}$), then
\begin{align} \label{qest1}
q(t, x) &= (1+o(1))\frac{\sqrt \pi}{32} (an)^{-7/2}b^2 (4u^2 + 12ut + 15t^2),
\end{align}
where we have defined the constants $a = \frac 12 p(1-p), b = \frac 16 (p - 3p^2 + 2p^3)$, and we have also defined
\begin{align} \label{udef}
u = u(t, x) = \frac ab n (x-tp).
\end{align}
For conciseness, define also $c = \frac{\sqrt \pi}{32} (an)^{-7/2}b^2$ to be the term in (\ref{qest1}) that does not depend on $x, t$ (though it does depend on $n$). Thus, if $x$ maximizes $\Pr(X = x)$ (or if $|x - tp| < n^{0.01}$), then
\begin{align} \label{qest}
q(t, x) &= c (1+o(1)) (4u^2 + 12ut + 15t^2).
\end{align}
This also implies that if $x$ maximizes $\Pr(X = x)$, then
\begin{align} \label{qestbd}
q(t, x) &\ge c (1+o(1)) \max(12u^2/5, 6t^2).
\end{align}
We are now equipped to show some results about the behavior of $\ell^*$ for large $n$. 
 
\begin{prop} \label{odddenom}
If $p = \frac rs$ for relatively prime positive integers $r, s$ and $s$ is odd, then for sufficiently large odd $n$, $\ell^* = \frac{n+s}2, x^* = r$.
\end{prop}
\begin{proof}
If $\ell = \frac{n+s}2$, then $t = s$, so it suffices to show that $q(t, x) \ge q(s, r)$ for all $x, t$. Suppose that for some $t, x$ we have $q(t, x) < q(s, r)$. We may also further assume that $x$ maximizes $\Pr(X = x)$, so that the estimate (\ref{qest}) holds. By (\ref{qest}), noting that $u(s, r) = 0$, 
\begin{align*}
q(s, r) = c (1+o(1)) s^2.
\end{align*}

Now, if $x \neq tp$, then $u(t, x) = \Omega(n)$, in which case by (\ref{qestbd}), $q(t, x) = \Omega(cn^2) > q(s, r)$.

On the other hand, if $x = tp$, then $t \ge s$, since $tp$ is an integer. If $t = s$, then $x = r$ and $q(t, x) = q(s, r)$. But if $t > s$, then by (\ref{qest}) $q(t, x) = c(1 + o(1))t^2 > q(s, r)$ for sufficiently large $n$.

Thus, if $n$ is sufficiently large, then $q(t, x) \ge q(s, r)$ for all $t, x$, as desired.
\end{proof}

\begin{prop} \label{evendenom}
If $p = \frac rs < 1$ for relatively prime positive integer $r, s$ and $s$ is even, then for odd $n$, $\ell^* = \left(\frac 12 + \frac{3}{5|1-2p|s} + o(1)\right)n$.
\end{prop}
\begin{proof}
We again maximize $q(t, x)$ according to the expression in (\ref{qest}).
 
There must be some odd congruency class $k$ such that if $t \equiv k \pmod s$, then there exists $x$ such that $x - tp = \pm \frac 1s$, where the sign is chosen to be the opposite of that of $\frac ab$, so that $u = - \frac {|a|n}{|b|s}$. We can further pick $t \equiv k \pmod s$ so that $t = -\frac 25 u + O(1)$. (Note that $-\frac 25 u = \frac{2|a|n}{5|b|s} = \frac{6n}{5|1-2p|s} = \frac{6n}{5|s-2r|} \le \frac{3n}{5}$ since $s$ is even and $p \neq \frac 12$, so $t$ is indeed between 0 and $n$ as required.) Then, by (\ref{qest}), for this choice of $t, x$, $q(t, x)$ is equal to
\begin{align} \label{opteven}
c(1+o(1)) \frac{12a^2n^2}{5b^2s^2}.
\end{align}

Now suppose that $t, x$ minimize $q(t, x)$. We then know that $q(t, x)$ is at most the quantity in (\ref{opteven}). Since $p$ is an integer multiple of $\frac 1s$, $u$ must be an integer multiple of $\frac{an}{bs}$ (by the definition of $u$ in (\ref{udef}). If $|u| \ge \frac{2|a|n}{|b|s}$, then by (\ref{qestbd}), $q(t, x) \ge c(1+o(1))\frac{48a^2n^2}{5b^2s^2}$, which is greater than (\ref{opteven}) for sufficiently large $n$. Thus, we must have $|u| = \frac{|a|n}{|b|s}$. If $u = \frac{|a|n}{|b|s}$, then all terms of $4u^2 + 12ut + 15t^2$ are positive, so the estimate (\ref{qest}) implies that $q(t, x) \ge c(1+o(1))(4u^2) = c(1+o(1))\frac{4a^2n^2}{b^2s^2}$, which is again greater than (\ref{qestbd}).

Thus we must have (still for sufficiently large $n$) $u = - \frac{|a|n}{|b|s}$. Now, again applying the estimate (\ref{qest}),
\begin{align*}
q(t, x) &= c (1+o(1)) \left(\frac{12}{5} u^2 + \left(t + \frac 25 u\right)^2 \right) \\
&= c (1+o(1)) \left(\frac{12a^2n^2}{5b^2s^2} + \left(t - \frac{2|a|n}{5|b|s}\right)^2. \right)
\end{align*}
Since this must be at most ($\ref{qestbd}$), we must have
\begin{align*}
t &= (1+o(1)) \frac{2|a|n}{5|b|s} \\
&= (1+o(1)) \frac{6n}{5|1-2p|s}.
\end{align*}
Thus,
\begin{align*}
\ell^* &= \frac{n+t}{2} \\
&= \left(\frac 12 + \frac{3}{5|1-2p|s} + o(1)\right)n,
\end{align*}
completing the proof.
\end{proof}

\begin{prop}
If $p$ is irrational, then for odd $n$, $\ell^* = \left(\frac 12 + o(1)\right)n$.
\end{prop}
\begin{proof}
Fix $\eps > 0$ ($\eps$ will be considered a constant for the purposes of asymptotic notation). Since the fractional part of $2kp + p$, as $k$ ranges over the nonnegative integers, is dense in $[0, 1)$, there exists some odd positive $t_\eps$ (which depends only on $\eps$ and $p$, not $n$) such that the fractional part of $t_\eps p$ is less than $\frac{|b|}{|a|}\eps$. We can thus find $x_\eps$ such that $|x_\eps - t_\eps p| < \frac{|b|}{|a|} \eps$. Then, $|u(t_\eps, x_\eps)| < \eps n$. Then, again applying the estimate (\ref{qest}),
\begin{align*}
q(t_\eps, x_\eps) &\le c (1 + o(1)) (10u(t_\eps, x_\eps)^2 + 21 t_\eps^2) \\
&\le c (1 + o(1)) (10 \eps^2 n^2 + 21 t_\eps^2).
\end{align*}
Now, let $t, x$ be the values that minimize $q(t, x)$. By (\ref{qestbd}),
\begin{align*}
q(t_\eps, x_\eps) &\ge q(t, x) \\
&\ge c (1 + o(1)) (6t^2).
\end{align*}
Combining these two inequalities,
\[c (1 + o(1)) (6t^2) \le c (1 + o(1)) (10 \eps^2 n^2 + 21 t_\eps^2),\]
so
\[t \le (1 + o(1)) \sqrt{\frac{1}{6}(10 \eps^2 n^2 + 21 t_\eps^2)}.\]
This is true for all $\eps$, so $t = o(n)$. Therefore, $\ell^* = \frac{n + t}{2} = \left(\frac 12 + o(1)\right)n$, as desired.
\end{proof}

\section{Concluding remarks} \label{concl}
We have shown that in order to maximize the concentration probability $\max_{x \in \R} \Pr(X = x)$, it must be the case that the $a_i$ are all 1 or $-1$. We then showed a sequence of results on $\ell^*$, the number of $a_i$ which are 1 in the maximal case. The results are summarized in the table below.

\begin{figure}[h] 
\centering
\bgroup
\def\arraystretch{1.3}
\begin{tabular}{|c|c|c|}
\hline
$n$ & $p$ & $\ell^*$ \\
\hline \hline
$n$ even & any $p$ & $\frac n2$ \\
\hline
$n$ odd, sufficiently large & $p = \frac rs$, $s$ odd & $\frac{n+s}2$ \\
\hline
$n$ odd, sufficiently large & $p = \frac rs$, $s$ even & $\left(\frac 12 + \frac{3}{5|1-2p|s} + o(1)\right)n$ \\
\hline
$n$ odd, sufficiently large & $p$ irrational & $\left(\frac 12 + o(1)\right)n$ \\
\hline
any $n$ & $p$ sufficiently small & $\left\lceil \frac n2 \right\rceil$ \\
\hline
\end{tabular}
\egroup
\caption{Results for $\ell^*$ in various cases}
\label{fig:lstarresults}
\end{figure}

Some questions remain. The problem of determining $\ell^*$ when $n$ is small is still open. Additionally, even when $n$ is large, there remain $o(1)$ terms when $p$ is not a rational with odd denominator. In particular, when $p$ is rational with even denominator, computer tests seem to indicate that for sufficiently large $n$, $\ell^*$ is the sum of a linear and a periodic function of $n$, allowing the exact value to be determined for sufficiently large $n$. We thus make the following conjecture, which is stronger than Proposition \ref{evendenom}.

\begin{conj}
If $p$ is rational with even denominator, then for sufficiently large odd $n$, the function $\ell^* - \left(\frac 12 + \frac{3}{5|1-2p|s}\right)n$ is periodic in $n$. In particular, $\ell^* = \left(\frac 12 + \frac{3}{5|1-2p|s}\right)n + O(1)$.
\end{conj}

It is not, however, possible to show a similar statement (that $\ell^* = \frac 12 n + O(1)$) for irrationals. Let
\[p = \sum_{i=0}^\infty \frac 1{F(i)},\]
where $F$ is a sufficiently fast-growing function so that $F(i)$ is an even positive integer for each $i$. Note that this is extremely well-approximated by a sequence of fractions with even denominator. We can then use methods similar to that of the proof of Proposition \ref{evendenom} to show a similar bound, in particular obtaining $\ell^* = \frac 12 n + \omega(1)$.

Another related line of questioning is a generalization to polynomials. One can interpret $X$ as a linear polynomial in the $\xi_i$. In the spirit of Theorem 1.10 of \cite{cacia}, we may ask a generalization of Question \ref{ques} where $X$ is an arbitrary polynomial in the $\xi_i$. Of course, some restriction on the polynomial is required to prevent it from being the zero polynomial. For example, Theorem 1.10 of \cite{cacia} restricts to only the point probabilities where $x$ is not equal to the constant coefficient of $X$. This is still only useful in the regime where $p$ is small, since otherwise the polynomial $(1-\xi_1)(1-\xi_2)\dots(1-\xi_n)$ is 0 with probability $1 - (1-p)^n$. But note that in this case $X$ usually does not depend on the value of any given $\xi_i$. Thus, a stronger condition is needed to ensure that $X$ has a strong enough dependence on each $\xi_i$. For example, we might require that flipping the value of $\xi_i$ for any given $i$ changes the value of $X$. This is also analogous to the linear case considered in this paper, where the requirement that $a_i \neq 0$ is the same as requiring that changing each $\xi_i$ changes $X$. We thus ask the following question.

\begin{question} \label{polyques}
Let $f$ be an $n$-variable polynomial such that
\[f(x_1, \dots, x_{i-1}, 0, x_{i+1}, \dots, x_n) \neq f(x_1, \dots, x_{i-1}, 1, x_{i+1}, \dots, x_n)\]
whenever $x_1, \dots, x_{i-1}, x_{i+1}, \dots, x_n$ are each $0$ or $1$. Then, let $X = f(\xi_1, \dots, \xi_n)$, where $\xi_1, \dots, \xi_n$ are independent instances of $\Ber(p)$. What is the maximum possible value of $\max_{x \in \R} \Pr(X = x)$?
\end{question} 

One important special case of the polynomial version is the concentration of the number of subgraphs isomorphic to a fixed graph in the random graph $G(n, p)$, which can be expressed as a polynomial in the indicator variables for the edges of the graph. (Note that this does not actually satisfy the condition in Question \ref{polyques}, since changing one edge may not change the subgraph count.) This question is considered by Fox, Kwan, and Sauermann in \cite{subgraph}.

\subsection*{Acknowledgments}
I want to thank Matthew Kwan for proposing this problem to me, mentoring me this past summer, and brainstorming with me on my different ideas towards this problem, as well as extensive help with editing this paper. Thanks also to Mehtaab Sawhney for pointing out the log-concavity fact in the proof of Lemma \ref{closetomean}, and to Zachary Chroman for useful discussions.

It was recently drawn to our attention that Ju{\v s}kevi{\v c}ius and Kurauskas have obtained some similar results \cite{arbdist}, including a version of Theorem \ref{mainthm} for large $n$. Their work is independent.

\bibliographystyle{amsplain}
\bibliography{main}

\begin{appendices}
\section{Approximating \texorpdfstring{$q(t, x)$}{q(t, x)}} \label{app}
In this appendix, we will estimate the value of $q(t, x)$ as defined in (\ref{qexp}) in order to help determine $\ell^*$. We assume that $p$ is not any of $0, \frac 12, 1$. As in Section \ref{odd}, we assume $p$ is constant for the purposes of asymptotic notation,

We will assume that $|x - tp| \le n^{0.01}$; the other case will be dealt with later. We also assume that $p$ is a fixed constant and that $n$ is sufficiently large.

We approximate the integrand in (\ref{qexp}), using a Taylor series approximation at $y=0$. We have 
\begin{align}
\ln \left| 1-p+pe^{-iy} \right| 
&= \frac 12 \ln \left( (1-p+pe^{-iy})(1-p+pe^{iy}) \right) \nonumber \\
&= \frac 12 \ln \left( 1 - (2 - e^{-iy} - e^{iy})p(1-p) \right) \nonumber \\
&= \frac 12 \ln( 1 - (2 - 2 \cos y)p(1-p) ) \label{lncos} \\
&= \frac 12 \ln( 1 - p(1-p)y^2 + O(y^4) ) \nonumber \\
&= -\frac{p(1-p)}2 y^2 + O(y^4). \label{part1}
\end{align}
The Taylor series approximations are valid since $y$ is bounded (by $\pi$).

Note that (\ref{lncos}) is decreasing for $0 \le y \le \pi$, and thus so is $| 1-p+pe^{-iy} |^n$. Thus, by (\ref{part1}), we have that 
\begin{equation}
\left| 1-p+pe^{-iy} \right|^n \le e^{n(-p(1-p)n^{-0.8}/2 + O(n^{-1.6}))} = e^{-\Omega(n^{0.2})}, \qquad \text{for $n^{-0.4} \le y \le \pi$}.
\end{equation}
Therefore, by (\ref{qexp}) (recalling that the integrand is positive),
\begin{align}
q(t, x) 
&= \int_{0}^{n^{-0.4}} 
\left| 1-p+pe^{-iy} \right|^n
\left(1-\cos \left( xy + t \arg(1-p+pe^{-iy}) \right)\right) 
\, dy \nonumber \\
&\qquad + \int_{n^{-0.4}}^{\pi} 
\left| 1-p+pe^{-iy} \right|^n
\left(1-\cos \left( xy + t \arg(1-p+pe^{-iy}) \right)\right) 
\, dy \nonumber \\
&= \int_{0}^{n^{-0.4}} 
\left| 1-p+pe^{-iy} \right|^n
\left(1-\cos \left( xy + t \arg(1-p+pe^{-iy}) \right)\right) 
\, dy + 2\pi e^{-\Omega(n^{0.2})}. \label{truncated}
\end{align}
We thus assume henceforth that $y \le n^{-0.4}$.

Now we also have (since $y$ is small)
\begin{align*}
\cos(xy + t \arg(1-p+pe^{-iy}))
&= \cos \left( xy + t \arctan\left( \frac{p \sin y}{1 - p + p \cos y} \right) \right). \\
\end{align*}
Expanding the Taylor series, we get
\begin{align}
\cos(xy + t \arg(1-p+pe^{-iy}))
&= \cos\left(xy + t\left(-py + \frac 16 (p - 3p^2 + 2p^3) y^3 + O(y^5)\right)\right) \nonumber \\
&= \cos \left( (x-tp)y + \frac 16 (p - 3p^2 + 2p^3) ty^3 + O(|t|y^5) \right) \nonumber \\
&= 1 - \frac 12 \left( (x-tp)y + \frac 16 (p - 3p^2 + 2p^3) ty^3 \right)^2 \nonumber \\
&\qquad + O((x-tp)^4 y^4 + |x-tp|^3 |t| y^6 + (x-tp)^2 t^2 y^8 \nonumber \\
&\qquad \qquad + |x-tp||t|^3y^{10} + |x-tp||t|y^6 + t^2y^8 + t^4 y^{12}) \nonumber \\
&= 1 - \frac 12 \left( (x-tp)y + \frac 16 (p - 3p^2 + 2p^3) ty^3 \right)^2 \nonumber \\
&\qquad + o((x-tp)^2y^2) + o(|t||x-tp|y^4) + o(t^2y^6),\label{part2}
\end{align}
where in the last step we have used the bounds $|x-tp| \le n^{0.01}, t \le n, y \le n^{-0.4}$. The first Taylor series approximation is again valid because $y$ is bounded, and the second is valid because, though its argument may be unbounded, the function $\cos$ itself is bounded. (In particular, it is true for all $z$ that $\cos z = 1 - z^2/2 + O(z^4)$.)

For brevity, define the constants $a = \frac 12 p(1-p), b = \frac 16 (p - 3p^2 + 2p^3)$. Note $a > 0, b \neq 0$. Applying the estimates from (\ref{part1}), (\ref{truncated}), and (\ref{part2}),
\begingroup
\allowdisplaybreaks
\begin{align*}
q(t, x) 
&= \int_{0}^{n^{-0.4}} 
e^{-any^2 + O(ny^4)}
\left( \frac 12 \left( (x-tp)y + b ty^3 \right)^2  + o((x-tp)^2y^2) + o(|t||x-tp|y^4) + o(t^2y^6)\right)
\, dy \nonumber \\*
&\qquad + 2\pi e^{-\Omega(n^{0.2})} \nonumber \\
&= \int_{0}^{n^{-0.4}} 
e^{-any^2 + O(n^{-0.6})}
\left( (1+o(1))\frac 12(x-tp)^2y^2 + (1+o(1))bt(x-tp)y^4  \right. \nonumber \\* &\qquad \qquad \qquad \qquad \qquad \qquad \qquad
\left.{} + (1+o(1))\frac 12 b^2t^2y^6 \right)
\, dy + 2\pi e^{-\Omega(n^{0.2})} \nonumber \\
&= \int_{0}^{n^{-0.4}} 
e^{-any^2}
\left( (1+o(1))\frac 12(x-tp)^2y^2 + (1+o(1))bt(x-tp)y^4 + (1+o(1))\frac 12 b^2t^2y^6 \right)
\, dy \nonumber \\*
&\qquad + 2\pi e^{-\Omega(n^{0.2})} \nonumber \\
&= (an)^{-1/2} \int_{0}^{n^{0.1}} 
e^{-z^2} \left((1+o(1)) \frac 12(x-tp)^2(an)^{-1}z^2 + (1+o(1)) bt(x-tp)(an)^{-2}z^4 \right. \nonumber \\*
&\qquad \qquad \qquad \qquad \qquad \qquad
\left. {} + (1+o(1)) \frac 12 b^2t^2(an)^{-3}z^6 \right) \, dz + 2\pi e^{-\Omega(n^{0.2})} \nonumber \\
&= (1+o(1))\frac 12 (an)^{-3/2} (x-tp)^2\int_{0}^{\infty} e^{-z^2} z^2 \, dz
+ (1+o(1)) (an)^{-5/2} bt(x-tp) \int_{0}^{\infty} e^{-z^2} z^4 \, dz \nonumber \\*
&\qquad + (1+o(1)) \frac 12 (an)^{-7/2} b^2t^2 \int_{0}^{\infty} e^{-z^2} z^6 \, dz + 2\pi e^{-\Omega(n^{0.2})} \nonumber \\
&=   (1 + o(1))\frac{\sqrt \pi}{8} (an)^{-3/2} (x-tp)^2 
+ (1 + o(1))\frac{3 \sqrt \pi}{8} (an)^{-5/2} bt(x-tp) \nonumber \\*
& \qquad + (1 + o(1))\frac{15 \sqrt \pi}{32} (an)^{-7/2} b^2t^2 \nonumber \\
&= \frac{\sqrt \pi}{32} (an)^{-7/2}b^2 ((1+o(1))4u^2 + (1+o(1))12ut + (1+o(1))15t^2),
\end{align*}
\endgroup
where we have defined
\begin{align*}
u = u(t, x) = \frac ab n (x-tp).
\end{align*}
Now, we have that $4u^2 + 12ut + 15t^2 \ge \max(\frac{12}{5}u^2, 6t^2)$, so we can pull the $1+o(1)$ factors to the front:
\begin{align}
q(t, x) &= (1+o(1))\frac{\sqrt \pi}{32} (an)^{-7/2}b^2 (4u^2 + 12ut + 15t^2), &\text{for $|x - tp| \le n^{0.01}$}. \label{qestapp}
\end{align}
This is the estimate we will use in order to determine $\ell^*$ in Section \ref{odd}.

It remains to rule out the case where $|x - tp| > n^{0.01}$. We thus show the following lemma.

\begin{lem} \label{closetomean}
If $x$ maximizes $\Pr(X = x)$ (for $t$ fixed), then $|x - tp| \le n^{0.01}$ for sufficiently large $n$.
\end{lem}
\begin{proof}
Let $x'$ be the value of $x$ which maximizes $\Pr(X = x)$. Suppose that $|x' - tp| > n^{0.01}$. Suppose also that $x' > tp$, so that $x' - tp > n^{0.01}$ (the other case is almost identical). 

Note that $X \sim \Bin(\ell, p) - \Bin(n - \ell, p) \sim \Bin(\ell, p) + \Bin(n - \ell, 1 - p) - (n - \ell)$. Since the binomial distribution is log-concave, and the convolution of two log-concave functions is also log-concave \cite{logconc}, the distribution of $X$ must also be log-concave. In particular, this means that the distribution of $X$ is unimodal. Thus, if we let $x_0 = \lfloor tp \rfloor$, then $\Pr(X = x_0) \le \Pr(X = x' - n^{0.01}) \le \Pr(X = x')$. Recalling from (\ref{qderiv}) that $\frac 1\pi q(t, x)$ and $\Pr(X = x)$ sum to a constant (depending only on $p$ and $n$), this means that $q(t, x_0) \ge q(t, x' - n^{0.01}) \ge q(t, x')$. However, we have
\begin{align*}
q(t, x' - n^{0.01}) + q(t, x') 
&\ge \int_0^{n^{-0.4}} \left| 1-p+pe^{-iy} \right|^n \left(2 - \cos \left( x'y + t \arg(1-p+pe^{-iy}) \right)\right. \\
&\qquad \qquad \qquad \qquad \qquad \qquad \qquad -{} \left. \cos \left( (x' - n^{0.01})y + t \arg(1-p+pe^{-iy}) \right)\right) \, dy \\
&= \int_0^{n^{-0.4}} e^{-any^2 + O(ny^4)} (2 -\cos(v) - \cos(v - n^{0.01}y))) \, dy \\
&= \Omega(1) \int_0^{n^{-0.4}} e^{-any^2} (2 -\cos(v) - \cos(v - n^{0.01}y))) \, dy,
\end{align*}
where $v$ is some function of $y$. But note that $n^{0.01}y = o(1)$ for $y \le n^{-0.4}$, so we have
\[2 -\cos(v) - \cos(v - n^{0.01}y) \ge \Omega(n^{0.02}y^2),\]
for any $v$, and thus
\begin{align*}
q(t, x' - n^{0.01}) + q(t, x') &\ge \Omega(n^{0.02}) \int_0^{n^{-0.4}} e^{-any^2} y^2 \, dy. \\
&= \Omega(n^{-1.48}) \int_0^{n^{0.1}} e^{z^2} z^2 \, dz \\
&= \Omega(n^{-1.48}) \int_0^{\infty} e^{z^2} z^2 \, dz \\
&= \Omega(n^{-1.48})
\end{align*}
However, by (\ref{qestapp}) (recalling that $|x_0 - tp| < 1$), we have that $q(t, x_0) = O(n^{-1.5})$, so for sufficiently large $n$, $q(t, x' - n^{0.01}) + q(t, x') > 2q(t, x_0)$, a contradiction.
\end{proof}
This completes the discussion of approximating $q(t, x)$.
\end{appendices}

\end{document}